\documentclass[11pt]{amsart}
\usepackage{amsmath,amssymb,amsthm}
\usepackage[latin1]{inputenc}
\usepackage{version,tabularx,multicol}
\usepackage{graphicx,float,psfrag}

\newcommand{\reff}[1]{(\ref{#1})}

\theoremstyle{plain}
\newtheorem{theo}{Theorem}[section]

\newtheorem{prop}[theo]{Proposition}
\newtheorem{lem}[theo]{Lemma}

\theoremstyle{remark}
\newtheorem{rem}[theo]{Remark}

\newcommand{\ca}{{\mathcal A}}

\newcommand{\ci}{{\mathcal I}}

\newcommand{\ck}{{\mathcal K}}

\newcommand{\cn}{{\mathcal N}}

\newcommand{\ct}{{\mathcal T}}

\newcommand{\E}{{\mathbb E}}

\newcommand{\N}{{\mathbb N}}
\renewcommand{\P}{{\mathbb P}}

\newcommand{\bm}{\mathbf m}

\newcommand{\ind}{{\bf 1}}

\newcommand{\Card}{{\rm Card}\;}

\newcommand{\val}[1]{\mathop{\left| #1 \right|}\nolimits}
\newcommand{\inv}[1]{\mathop{\frac{1}{ #1}}\nolimits}
\newcommand{\expp}[1]{\mathop {\mathrm{e}^{ #1}}}

\newcommand{\lb}{[\![}
\newcommand{\rb}{]\!]}

\title{A construction of a $\beta$-coalescent via the pruning of
  Binary Trees}

\date{\today}
\author{Romain Abraham} 

\address{
Romain Abraham,
Laboratoire MAPMO, CNRS, UMR 6628,
F\'ed\'eration Denis Poisson, FR 2964,
 Université d'Orléans,
B.P. 6759,
45067 Orléans cedex 2,
France.
}
  
\email{romain.abraham@univ-orleans.fr}

\author{Jean-François Delmas}

\address{
Jean-Fran\c cois Delmas,
Université Paris-Est, \'Ecole des Ponts, CERMICS, 6-8
av. Blaise Pascal, 
  Champs-sur-Marne, 77455 Marne La Vallée, France.}

\email{delmas@cermics.enpc.fr}

\thanks{This work is partially supported by the ``Agence Nationale de
  la Recherche'', ANR-08-BLAN-0190.}

\begin{document}

\begin{abstract}
Considering a random binary tree with $n$ labelled leaves, we use a
pruning procedure on this tree in order to construct a $\beta(\frac{3}{2},\frac{1}{2})$-coalescent
process. We also use the continuous analogue of this construction, i.e. a
pruning procedure on Aldous's continuum random tree, to construct a
continuous state space process that has the same structure as the
$\beta$-coalescent process up to some time change. These two
constructions enable us to obtain results on the coalescent process
such as the asymptotics on the number of coalescent events or the law
of the blocks involved in the last coalescent event.
\end{abstract}

\maketitle

\section{Introduction}

Let $\Lambda$ be a   finite  measure  on
$[0,1]$. 
A $\Lambda$-coalescent $(\Pi(t),t\ge 0)$ is a Markov process which takes
values in the set of partitions of $\N^*=\{1, 2, \ldots\}$ introduced in
\cite{p:cmc}. It is defined via  the transition rates of its restriction
$(\Pi^{[n]}(t),t\ge 0)$ to the  $n$ first integers: if $\Pi^{[n]}(t)$ is
composed of $b$  blocks, then $k$ $(2\le k\le  b)$ fixed blocks coalesce
at rate:
\begin{equation}
   \label{eq:lambda-bn}
\lambda_{b,k}=\int_0^1u^{k-2}(1-u)^{b-k}\Lambda (du).
\end{equation}
In particular a coalescent event arrives at rate:
\begin{equation}
   \label{eq:lambda-n}
\lambda_b=\sum_{k=2}^b \binom{b}{k} \lambda_{b,k}. 
\end{equation}

 As examples  of $\Lambda$-coalescents,  let us  cite Kingman's
coalescent       ($\Lambda(dx)=\delta_0(dx)$,   see   \cite{k:c}),      the
Bolthausen-Sznitman     coalescent     ($\Lambda(dx)=\ind_{(0,1)}(x)dx$, see
\cite{bs:rpcacm}),   or   $\beta$-coalescents   ($\Lambda(dx)$  is   the
$\beta(2-\alpha,\alpha)$       distribution      with      $0<\alpha<
2$, see \cite{bbcemsw:asbbc, bbs:bcsrt},  or  the
$\beta(2-\alpha,\alpha-1)$  distribution 
with  $1<\alpha<  2$, see \cite{fh:cbi-gfvi}).  We  refer to  the  survey
\cite{b:rpct} for  further results on coalescent processes.  

The goal of this paper is to give a new representation for the
$\beta(3/2,1/2)$-coalescent using the pruning of random binary trees. This
kind of idea has already been used in \cite{gm:rrtbsc} where the
Bolthausen-Sznitman coalescent is constructed via the cutting of
a random recursive tree. 

\subsection{Pruning of binary trees}
Now we describe the coalescent associated with
the pruning of the random binary tree. 

We recall the normalization constant in the beta distribution for $a>0$, $b>0$:
\[
\beta(a,b)=\frac{\Gamma(a)\Gamma(b)}{\Gamma(a+b)}=\int_{(0,1)} x^{a-1}
(1-x)^{b-1}\, dx. 
\]

Fix an
integer $n$ and consider a uniform random ordered binary tree with $n$
leaves ($2n-1$ vertices plus a root; $2n-1$ edges). Label these leaves from 1 to $n$ uniformly at random. After an
exponential time of parameter: 
\begin{equation}
   \label{eq:lambda-3/2}
\lambda_n= (n-1) \;
\beta\left(\frac{1}{2},n-\frac{1}{2}\right),
\end{equation}
 choose one of the $n-1$ inner vertices of
the tree uniformly at random, coalesce all the leaves of the
subtree attached at the chosen node and remove that subtree from the
original tree. Restart then the process with the resulting tree (using
$\lambda_k$ as the new parameter of time, where $k$ is the new number
of leaves) until all the leaves coalesce into a single one.

Figure \ref{pic:ici} gives an example of such a coalescence for
$n=5$.

\begin{center}
\begin{figure}[H]
\includegraphics[width=8cm]{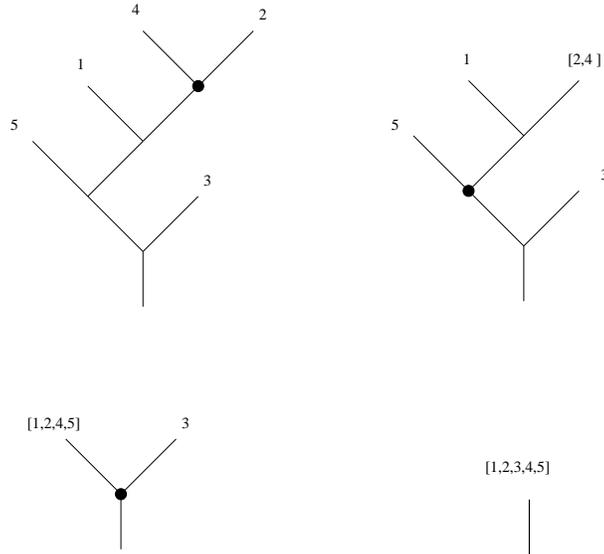}
\caption{An example of the coalescence  of a binary tree for $n=5$.  The
  selected  node is  in  bold and  the  times between  each picture  are
  exponentially  distributed   with  respective  parameter  $\lambda_5$,
  $\lambda_4$, $\lambda_2$.}
\label{pic:ici}
\end{figure}
\end{center}

This defines a process $(\Pi^{[n]}(t),t\ge 0)$. The main result of the
paper is the following Theorem.

\begin{theo}\label{thm:main}
The process $(\Pi^{[n]}(t),t\ge 0)$ defined as  the coalescent associated with
the pruning of the random binary tree is the restriction to $\{1, \ldots,
  n\}$ of a $\beta(3/2,1/2)$-coalescent with coalescent measure:
\begin{equation}
   \label{eq:Lambda-3/2}
\Lambda(du) =\sqrt{\frac{u}{1-u}}\; du.
\end{equation}
\end{theo}

Let us remark that the $\beta$-coalescents introduced in
\cite{bbcemsw:asbbc} and usually studied are
$\beta(2-\alpha,\alpha)$-coalescents with $1<\alpha<2$. Here we have
$\alpha=1/2$ and this does not enter the usual case. The reason is
that for $1<\alpha<2$, the $\beta$-coalescent comes down from infinity
($\Pi(t)$ has a.s. a finite number of blocks for any $t>0$)
which is not the case for $\alpha=1/2$ according to the criterion from
\cite{b:rpct} Theorem 3.5 or \cite{p:cmc}  as:
\[
\int_0^1\frac{\Lambda(du)}{u}<+\infty.
\]

\subsection{Pruning of Aldous's CRT}

A coalescent process may also be viewed as a process $(\ci(t),t\ge 0)$ taking
values  in  the  interval  partitions  of $[0,1]$. The length of each
interval represents the mass of a block, and the process represents
blocks (whose sizes sum to 1) that merge together as time passes. We
can go from this interval partition-valued process to the previous
$\Lambda$-coalescent framework by the classical paintbox procedure
(see for instance
\cite{p:csp}): consider $(U_i)_{i\in\N}$ a sequence of i.i.d. uniform
random variables on $[0,1]$, independent of the process $\ci$ and, for
every $t\ge 0$, say that the integers $i$ and $j$ belong to the
same block in $\Pi(t)$ if and only if $U_i$ and $U_j$ belong to the
same interval in $\ci(t)$.

The continuous analogue of the  binary tree is Aldous's continuum random
tree (CRT) which can be obtained  as the limit (in an appropriate sense)
of the rescaled  uniform binary tree with $n$ leaves  when the number of
leaves $n$ tends to infinity. A  pruning theory of such a continuum tree
has been introduced in \cite{as:psf} (see \cite{adv:plcrt} for a general
theory of the  pruning of L\'evy trees) and will  be recalled in Section
\ref{sec:crt}. Using  this pruning procedure,  we are able to  define an
interval-partition-valued process in Section \ref{sec:pruning} which has
the  same structure  as the  $\beta(3/2,1/2)$-coalescent except  for the
times (when  sampling $n$ points  uniformly distributed on  $(0,1)$, the
time   interval   between   two   coalescences  is   not   exponentially
distributed).  However, we conjecture that an appropriate change of time
would be enough so  that the interval-partition-valued process is really
associated  with   the  $\beta(3/2,1/2)$-coalescent  via   the  paintbox
procedure. This gives a nice  interpretation of the dust (or fraction of
singletons) in the $\beta(3/2,1/2)$-coalescent.

\subsection{Number of coalescent events and last coalescent event}

The construction using discrete trees allows us to recover in Section \ref{sec:appli-X} the
asymptotic  distribution of  the  number of  coalescent  events given  by
\cite{gim:ldc} in a more general framework, see  also \cite{im:aecmc}. 
\begin{prop}
\label{prop:cv-Xn}
 Let $X'_n$
be the number of collisions undergone by $(\Pi^{[n]}(t),t\ge 0)$. Then
we have:
\[
\frac{X'_n}{\sqrt{n}}\overset{(d)}{\underset{n\to\infty}{\longrightarrow}}
\sqrt{2} Z,
\]
where $Z$ has a Rayleigh distribution with density
$x\expp{-x^2/2}\ind_{\{x>0\}}$. 
\end{prop}

Let us remark that according to
Section 5.4 in \cite{gim:ldc} $Z$ is distributed as:
\[
\frac{2}{\sqrt{\pi}}
\int_0^\infty dt \;  \expp {-S_t/2},
\] where $(S_t, t\geq 0)$ is a
subordinator with $\E[\exp(-\lambda S_t)]= \exp(- t\Phi(\lambda))$ where
for $\lambda>0$:
\[
\Phi(\lambda)=\int_0^1 \!\left(1-(1-u)^\lambda\right)
\frac{\Lambda(du)}{u^2} = \int_0^1\! \left(1-(1-u)^\lambda\right)
u^{-3/2}(1-u)^{-1/2} 
du = 2\sqrt{\pi} \; \frac{\Gamma\left(\lambda +\inv{2}
  \right)}{\Gamma(\lambda)}\cdot
\]
(See \cite{by:eflp} for more results on exponentials of Lévy
processes.)



The continuum tree  construction allows us to study  the last coalescent
event   (see   \cite{gm:rrtbsc}    for   similar   questions   for   the
Bolthausen-Sznitman  coalescent, and  \cite{dc:spsmtsbp}  for stationary
CSBP).  In Section \ref{sec:appli-last}, we consider the number $E_n$ of
external branches or singletons involved in the last coalescent event as
well  as the  number of  blocks $B_n$  involved in  the  last coalescent
event.

\begin{prop}
   \label{prop:Ext-B}
   Let $B_n$ be the number of blocks and $E_n$ be the number of  singletons 
   involved  in   the  last  coalescent   event  of  $(\Pi^{[n]}(t),t\ge
   0)$. Then we have:
\[
(B_n, E_n)\overset{(d)}{\underset{n\to\infty}{\longrightarrow}}
(B,E),
\]
where  $(B-E, E)$ are  finite random  variables with  generating function
$\Phi$ given, for $\rho,\rho_*\in [0,1]$ by:
\begin{align}
\nonumber
\Phi(\rho,\rho_*)
&=\E\left[\rho^{B-E} \rho_*^E\right]\\
\label{eq:Phi}
&=
\rho\left(1+ \log(2)- \log\left(1+\sqrt{1-\rho_*}
    -\frac{\rho+\rho_*}{2}\right) \right) 
-\sqrt{\rho}  \log\left(\frac{1+\sqrt{1-\rho_*}+\sqrt{\rho}}
{1+\sqrt{1-\rho_*}-\sqrt{\rho}}
\right). 
\end{align}
Furthermore $B-E$ is stochastically less than (or equal to) $E+1$. 
\end{prop}
Notice  that $\Phi(1,1)=1$  which indeed  implies  that $B$  and $E$  are
finite, that is $B_n$ and $E_n$  are of order $1$. However,  $B-E$ and
$E$ have infinite expectation, as this can be checked from their
generating functions given below.

The generating function of $E$ is given by:
\[
\E\left[\rho_*^E\right]=\Phi(1, \rho_*)=1- 2
\log\left(1+\frac{\sqrt{1-\rho_*}}{2}\right),
\]
and the generating function of $B-E$ is given by:
\[
\E\left[\rho^{B-E}\right]
=\Phi(\rho,1) =
\rho\left(1+ \log(4)- \log\left(1 -\rho\right) \right)
-\sqrt{\rho}  \log\left(\frac{1+\sqrt{\rho}}
{1-\sqrt{\rho}}
\right). 
\]
There is a nice interpretation of the distribution of $B-E-1$ given
after Proposition \ref{prop:UVW}.
The generating function of $B$ is given by:
\[
\E\left[\rho^B\right]
=\Phi(\rho,\rho) 
=
\rho\left(1+ \log(2)- \log\left(1+\sqrt{1-\rho} -\rho\right) \right)
-\sqrt{\rho}  \log\left(\frac{1+\sqrt{1-\rho}+\sqrt{\rho}}
{1+\sqrt{1-\rho}-\sqrt{\rho}}
\right). 
\]
Of course we have a.s. $B\geq 2$.

\begin{rem}
   \label{rem:calcul}
We can compute various quantities related to $E$ and $B$. 
We have:
\begin{align*}  
&\P(E=0)=1- 2\log\left(\frac{3}{2}\right)\simeq 0.19, \quad
\P(E=1)=\inv{3},\quad
\P(E=2)=\frac{1}{9},\\
&\P(B-E=0)=0, \quad
\P(B-E=1)=\log(4)-1\simeq 0.39,\quad
\P(B-E=2)=\frac{1}{3},\\
&\P(B=0)=\P(B=1)=0, \quad
\P(B=2)=\frac{5}{12},\quad
\P(E=3)=\frac{23}{160}\cdot
\end{align*}
In  particular,  we  have  $\P(E>5)\leq 25\%$,  $\P(B>5)\leq  32\%$  and
$\P(B-E>5)\leq 11\%$.
\end{rem}

The paper is organized as follows. The proof of Theorem \ref{thm:main} is
given in Section \ref{sec:proof-main}. The link with Aldous' CRT,
presented in Section \ref{sec:crt}, is given in Section \ref{sec:pruning}
using a pruning procedure; the reduced sub-trees are presented in Section
\ref{sec:reduced}. A proof and a comment on Proposition \ref{prop:cv-Xn} are
given in Section \ref{sec:appli-X}. Proposition \ref{prop:Ext-B}
is proved in Sections \ref{sec:proof-VB} and \ref{sec:proof-UVW}. 


\section{The $\beta(3/2,1/2)$-coalescent}
\label{sec:proof-main}
In order to prove Theorem \ref{thm:main}, we first need to compute the
rates $\lambda_{n,k}$ at which $k$ given blocks among $n$ blocks
coalesce. This is the purpose of the next Proposition.

\begin{prop}
\label{prop:l-nk}
  For the  coalescent of the random  binary tree, we have  for any $2\le
  k\le n$,
\[
\lambda_{n,k}=\beta\left(k-\frac{1}{2},n-k+\frac{1}{2}\right). 
\]
\end{prop}
We recall the duplication formula for $a>0$:
\[
\Gamma\left(a+\inv{2}\right)=\frac{\sqrt{\pi}}{2^{2a -1}}
\frac{\Gamma(2a)}{\Gamma(a)}. 
\]

\begin{proof}
Let us first remark that, by construction, since there are $n-1$
internal vertices, we have:
\[
\lambda_{n,n}= \frac{\lambda_n}{n-1}=
\beta\left(n-\frac{1}{2},\frac{1}{2}\right)= \sqrt{\pi}\ \frac{\Gamma(n-\inv{2 })}{\Gamma(n)}\cdot
\]

It is well known that the number of ordered binary trees with $n$
leaves is given by the Catalan numbers:
\[
b_n=\frac{1}{n}\left(
\begin{array}{c}2n-2\\n-1\end{array}
\right)=\frac{(2n-2)!}{(n-1)!n!}\cdot
\]
Hence the number of ordered binary trees with $n$
labelled leaves is:
\begin{equation}
   \label{eq:def-C}
C_n=n!b_n=\frac{(2n-2)!}{(n-1)!}=\frac{2^{2n-2}}
{\sqrt\pi}\Gamma(n-\frac{1}{2}). 
\end{equation}

Consider  a binary  tree  with  $n$ labelled  leaves.  Let $2\leq  k\leq
n$. Fix $k$  labels, say the $k$ first. For these  labels to coalesce at
the same  time, the leaves with these  $k$ labels must exactly  lie in a
single subtree of the initial tree. Therefore, to construct such a tree,
we must choose an ordered binary tree with $k$ leaves labelled from 1 to
$k$  ($C_k$ possibilities), choose  an ordered  binary tree  with $n-k+1$
leaves labelled  from $k$ to  $n$ ($C_{n-k+1}$ possibilities)  and graft
the tree with $k$  leaves  at the leaf  labelled $k$. Then, for  the $k$ first
labels to coalesce, the chosen branch  must be the branch that links the
two     subtrees    (and     each    branch     is    cut     at    rate
$\lambda_{n,n}$). Therefore, we have: 
\begin{align*}
\lambda_{n,k} 
& =\frac{C_kC_{n-k+1}}{C_n}\lambda_{n,n}\\
&=\frac{2^{2k-2}}{\sqrt\pi}\Gamma\left(k-\inv{2}\right)\frac{2^{2n-2k}}{\sqrt\pi}\Gamma\left(n-k+\inv{2}\right)\frac{\sqrt\pi}{2^{2n-2}}\inv{\Gamma\left(n-\inv{2}\right)}\,
\sqrt\pi \, \frac{\Gamma\left(n-\inv{2}\right)}{\Gamma(n)}\\
&=\beta\left(
k-\inv{2},
n-k+\inv{2}\right). 
\end{align*}
\end{proof}

This proves that the process evolves like a $\Lambda$-coalescent with
$\Lambda$ given by \reff{eq:Lambda-3/2} up to the time of the first
merger. To finish the proof of Theorem \ref{thm:main}, it remains to
prove that, after that first merger, the resulting tree is still a
uniform binary tree with uniform labeled leaves.

Let us fix $k\le n$ and let $T_k$ be a tree with $k$ labelled leaves,
one being labelled by the block $[i_1,\ldots,i_{n-k+1}]$, the others
being labelled by singletons. We want to compute the probability of
obtaining that tree after the first merger. From this tree, we
construct a tree with $n$ leaves by grafting, on the leaf of $T_k$
labelled by the block $[i_1,\ldots,i_{n-k+1}]$, a tree with $n-k+1$ leaves labelled by
$\{i_1,\ldots,i_{n-k+1}\}$. There are exactly $C_{n-k+1}$ different
trees (this corresponds to the choice of the grafted tree). Moreover,
the tree $T_k$ is obtained after the first merger if the original tree
is one of those, and if the chosen internal node  is the
leaf labelled by the block. Hence, the probability of obtaining $T_k$
is
$$\frac{1}{n-1}\frac{C_{n-k+1}}{C_n}\cdot$$
We note that this probability depends only of the number $k$ of leaves of
the tree and not on the tree itself, hence conditionally on merging
$n-k+1$ leaves, the resulting tree is
still uniform among all the trees with $k$ leaves.


\section{Links with the pruning of Aldous's CRT}
\label{sec:aldous}

\subsection{Aldous's CRT}
\label{sec:crt}

In \cite{a:crt1}, Aldous introduced  a continuum random tree (CRT) which
can be  obtained as  the scaling limit  of critical  Galton-Watson trees
when the length of  the branches tends to 0. This tree  can also be seen
as  a  compact (with  respect  to  the  Gromov-Hausdorff topology)  real
tree. Indeed,  a real  tree is a  metric space $(\ct,d)$  satisfying the
following two properties for every $x,y\in\ct$:
\begin{itemize}
\item (unique geodesic) There is a unique isometric map $f_{x,y}$ from
  $[0,d(x,y)]$ into $\ct$ such that $f_{x,y}(0)=x$ and
  $f_{x,y}(d(x,y))=y$.
\item (no loop) If $\varphi$ is a continuous injective map from
  $[0,1]$ into $\ct$ such that $\varphi(0)=x$ and $\varphi(1)=y$, then
$$\varphi([0,1])=f_{x,y}([0,d(x,y)]).$$
\end{itemize}
A rooted real tree is a real tree with a distinguished vertex denoted
$\emptyset$ and called the root.

It is well known that every compact real tree can be coded by a
continuous function (this coding is described below), Aldous's tree is
just the real tree coded by a normalized Brownian excursion $e$.
We refer to
\cite{lg:rrt} for more details on real trees and their coding by
continuous functions.

\begin{rem}
In fact Aldous' CRT is coded by $2e$. We omit here the factor 2 for
convenience but some constants may vary between this paper and Aldous'
results. Our setting corresponds to the branching mechanism
$\psi(\lambda)=2\lambda^2$ in \cite{dl:rtlpsbp}.
\end{rem}

Let $e$ be a normalized Brownian excursion on $[0,1]$. For
$s,t\in[0,1]$ we set:
\[
d(s,t)=e(s)+e(t)-2\inf_{u\in[s\wedge t,s\vee t]}e(u).
\]
We then define the equivalence relation $s\sim t$ iff $d(s,t)=0$ and the
tree $\ct$  is the quotient  space $\ct=[0,1]/\sim$. We denote  by $p$
the  canonical  projection from  $[0,1]$  to  $\ct$.   The distance  $d$
induces a distance on $\ct$ and  we keep notation $d$ for this distance.
The  metric  space  $(\ct,d)$ is  then  a  real  tree. The metric
space $(\ct,d)$ can be seen as a rooted tree by choosing $\emptyset=p(0)$ as the root.

For  $x,y\in\ct$, we  denote by  $\lb x,y\rb$  the range  of  the unique
injective continuous path between $x$ and $y$ in $\ct$. We also define a
length measure $\ell(dx)$ on the skeleton of $\ct$ (i.e. non-leaves
vertices) by
$$\ell(\lb x,y\rb)=d(x,y).$$
Finally, for $x,y\in\ct$, we denote by $a(x,y)$ their last common
ancestor i.e. the unique point in $\ct$ such that:
\[
\lb \emptyset,x\rb\cap \lb \emptyset,y\rb =\lb \emptyset,a(x,y)\rb.
\]
For simplicity, we write for $s,t\in[0,1]$, $a(s,t)$ for $a(p(s),p(t))$.

\subsection{The interval-partition-valued process}
\label{sec:pruning}

As in \cite{as:psf,ad:rpcrt}, we throw points on the CRT ``uniformly''
on the skeleton of the CRT and add more and more points as time goes by. More
precisely, we consider a Poisson point measure $M(d\theta, dx)$ on
$[0,+\infty)\times \ct$ with intensity $4\,d\theta \ell(dx)$.

For $\theta>0$, we define an equivalence  relation  $\sim_\theta$ on
$[0,1]$ by:
$$s\sim_\theta t\iff \begin{cases}
s=t & \mbox{or}\\
M([0,\theta]\times \lb \emptyset,a(s,t)\rb)>0.
\end{cases}$$
We   set  $(I_k^\theta,  k\in   \ck_\theta)$  the   equivalence  classes
associated with $\sim_\theta$ non-reduced  to a singleton. Let us remark
that each  $I_k^\theta$ is an  interval. 

Equivalently, we define
$$\ct_\theta=\{x\in\ct,\ M([0,\theta]\times\lb\emptyset,x\rb)=0\}$$
which   is  the  set   of  vertices   that  have   no  marks   on  their
linage. The tree $\ct_\theta$ is  called the pruned tree; it corresponds
to the whole dust of the coalescent process. Then consider the
set  $(\ct_k^\theta,k\in  \ck_\theta)$ of  the  connected components  of
$\ct\setminus \ct_\theta$  which are the  sub-trees that are  grafted on
the    leaves   of    $\ct_\theta$    to   get    $\ct$   (see    Figure
\ref{fig:coal-tree}). Then  $I_k^\theta$ is just the  set of $s\in[0,1]$
such that $p(s)\in \ct_k^\theta$.

\begin{center}
\begin{figure}[H]
\psfrag{Ttheta}{$\ct_\theta$}
\psfrag{T1theta}{$\ct_1^\theta$}
\psfrag{T2theta}{$\ct_2^\theta$}
\psfrag{T3theta}{$\ct_3^\theta$}
\includegraphics[width=10cm]{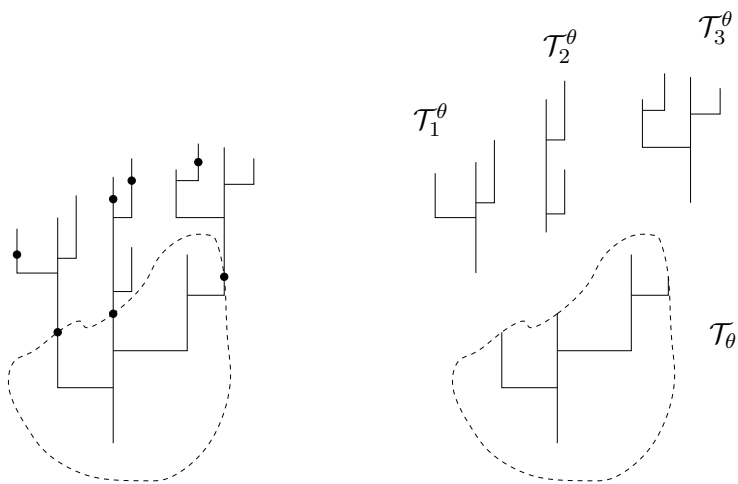}
\caption{Left: Aldous's CRT with the marks. Right: the sub-trees
  constructed from the marks}
\label{fig:coal-tree}
\end{figure}
\end{center}

By the  definition of
the mark process, for $\theta'>\theta$ we have
$$s\sim_{\theta}t\Longrightarrow s\sim_{\theta'} t$$
and consequently
$$\forall k\in\ck_{\theta},\ \exists
k'\in\ck_{\theta'},\ I_k^\theta\subset I_{k'}^{\theta'}.$$

Therefore,   the   process   $\ci=\Big((I_k^\theta,  k\in   \ck_\theta),
\theta\ge  0\Big)$ can  be  viewed  as a  process  where several  blocks
coalesce together (with part of the dust) into a single larger
block. On the picture on trees, when $\theta$ increases, the number of
marks also increases and when a mark appears on $\ct_\theta$, some
sub-trees above $\ct_\theta$ coalesce with part of $\ct_\theta$.

Let us remark that, as announced in the introduction, this process
always has dust which corresponds to  individuals that have no marks on their
lineage i.e. that belong to $\ct_\theta$ . The dust has Lebesgue measure $\sigma_\theta$:
\[
\sigma_\theta=\int_0^1 \ind_{\{ M([0,\theta]\times \lb \emptyset,p(s)\rb)=0\}}\;
ds=\int_0^1 \ind_{\{p(s)\in \ct_\theta\}}ds.
\]
We  recall  the distribution  of  $(\sigma_\theta,  \theta\geq 0)$  from
\cite{ap:sac}  on  the   size  process  of  a  tagged   fragment  for  a
self-similar fragmentation, see  also \cite{ad:ctvmp} Proposition 9.1 or
Corollary   9.2    (but   with   $\beta=2$).     The   distribution   of
$(\sigma_\theta, \theta\geq 0)$  under the normalized Brownian excursion
measure   is  given  by   $(1/(1+4\tau_\theta),  \theta\geq   0)$  where
$(\tau_\theta, \theta\geq  0)$ is a  stable subordinator with  index 1/2
with no  drift, no killing, and  Lévy measure $(2\pi  x^3)^{-1/2} dx$ on
$(0, \infty  )$: for $\lambda\geq  0$, $\E[\expp{-\lambda \tau_\theta}]=
\expp{-\theta\sqrt{2\lambda}}$.

\subsection{The reduced tree with $n$ leaves}
\label{sec:reduced}
We  apply  the  paintbox   procedure  to  this  process.   Consider  $n$
independent (and independent of the process) random variables, uniformly
distributed  on  $(0,1)$.   This  corresponds  to  choosing  $n$  leaves
uniformly on  Aldous's CRT. Then the  law of the  reduced tree, $\ct^n$,
containing these $n$ leaves  is given in \cite{dl:rtlpsbp}, Section 3.3,
or  \cite{a:crt3}, Section  4.3.  The  shape of  the reduced  tree  is a
binary tree with uniform probability on all ordered binary tree with $n$
leaves.  As the  tree is binary, it is composed  of $2n-1$ branches with
lengths $(h_1, \ldots, h_{2n-1})$ and distribution:
\[
 2^{n+1} \;  \frac{(2n-1)!}{(n-1)!}\;  s_n \expp{-2 s_n^2} \; \ind_{\{h_1>0,
  \ldots, h_{2n-1}>0\}}\; dh_1\cdots
dh_{2n-1},
\]
where  $s_n=\sum_{k=1}^{2n-1} h_k$ is  the total  length of  the reduced
tree.

Since  the reduced  tree is  binary and  all its  edges  are identically
distributed,  the  first  mark  that  appears on  the  reduced  tree  is
uniformly  distributed among  all  the  edges and  we  deduce that  this
continuous   coalescent  procedure   has   the  same   structure  as   a
$\beta(3/2,1/2)$-coalescent  process.   However   this  process  is  not
\textit{stricto sensu}  a coalescent  process as the  time at  which $n$
leaves  chosen at  random  undergo a  coalescence  is not  exponentially
distributed,  but is  distributed according  to an  exponential random
variable with  (random) parameter $H_n$,  with $H_n$ equal to  $4$ times
the  total  length of  the  internal  branches  length.  Thus  $H_n$  is
distributed as  $4 \sum_{k=1}^{n-1}  h_k$.  Notice the  random variables
$(h_1,  \ldots,   h_{2n-1})$  are  exchangeable   and  $\E[h_1]=2^{-3/2}
\Gamma(n-1/2)/\Gamma(n)$. In particular, we have:
\[
\E[H_n]= 4 (n-1) \E[h_1]=\sqrt{2}(n-1) 
\frac{\Gamma(n-\inv{2})}{\Gamma(n)} = \sqrt{\frac{2}{\pi}}
\lambda_n.
\]
So $\E[H_n]$  corresponds (up to a  scaling constant) to the  rate of the
$\beta(3/2,1/2)$-coalescent starting with $n$ individuals.

\begin{rem}
 We conjecture there exists a random time change $(A_t, t\geq 0)$ such that the process
  process $\Big((I_k^{A_t}, k\in  \ck_{A_t}), t\ge 0\Big)$ is the
  interval-partition-valued process associated with the
  $\beta$-coalescent. However, we were not able to exhibit such a time change.
\end{rem}


\section{Applications}
\label{sec:appli}

\subsection{Number of coalescent events}
\label{sec:appli-X}

Proposition \ref{prop:cv-Xn}  is just a consequence of Theorem 6.2
of \cite{j:rcrdrt} on  the number of cuts used to isolate  the root in a
Galton-Watson tree with a given number of leaves. We must just remark  that a binary
tree with $n$ leaves has $2n-1$ edges and that a Galton-Watson tree with binary
branching conditioned to have $n$ leaves is uniformly distributed among
the binary trees with $n$ leaves.

\begin{rem}
   \label{rem:pruning-X}
According  to \cite{ad:rpcrt}, we also have the following equality in
distribution: 
\[
 Z \stackrel{(d)}{=} \Theta
\quad\text{with}\quad
\Theta=\int_0^\infty \sigma_\theta\; d\theta.
\]

Notice that if a pruning mark appears twice or more on the same internal
branch, only one  will be taken into account as  a coalescent event, and
that the pruning  marks which appear on the external  branch will not be
taken into  account as a coalescent  event.  Let $X_n$ be  the number of
pruning events of the reduced tree with $n$ leaves and $X'_n$ the number
of  coalescent  events. We  deduce  that  $X_n$ is stochastically larger
than (or equal to) $X'_n$. But  the  a.s.
convergence which appears in  \cite{ad:rpcrt} (see also \cite{h:tcl} for
the fluctuations) gives that a.s.:
\[
\lim_{n\rightarrow+\infty }
\frac{X_n}{\sqrt{n}}=\sqrt{2}\Theta.
\]
This implies that  the number of
marked external branches, say $W_n$, is of order $o(n)$. We shall see in
Section \ref{sec:proof-VB} that it is in fact of order $1$. 
\end{rem}

\subsection{Last coalescent event}
\label{sec:appli-last}
\subsubsection{The CRT framework}
\label{sec:CRT-frame}

We refer  to \cite{dl:rtlpsbp} for  the construction of Lévy trees and
their main properties.
Let  $\ct$  be  a  continuum  Lévy tree  associated  with  the  branching
mechanism $\psi(\lambda)=\alpha\lambda^2$, i.e. coded as in
Section \ref{sec:crt} by a positive excursion $e$ of $\sqrt {2\alpha}B$ where
$B$ is a standard Brownian motion. We denote by $\N$ the ``law''  of this
tree when the coded function $e$ is distributed according to the It\^o
measure (hence $\N$ is an infinite measure and is not really a distribution) and by $\N^{(r)}$ the same law when $e$ is distributed as a
normalized  excursion of length $r$. We denote by $\bm$ the mass measure
on the tree, i.e. the image of the Lebesgue measure on $[0,+\infty)$
  by the canonical projection $p$.

Conditionally   on   $\ct$,  let   $M$   be   defined   as  in   Section
\ref{sec:pruning} with the intensity $2\alpha\,d\theta\, \ell(ds)$ instead of
$4\,d\theta\,  \ell(ds)$. (We introduce  the parameter  $\alpha$ in  order to
make the references to \cite{ad:ctvmp} easier.)  Consider the pruning of
the tree $\ct$ at time $\theta>0$:
\[
\ct_\theta=\{s\in \ct; \, M([0,\theta]\times \lb \emptyset,s\rb)=0\}.
\]
And we set $\sigma_\theta=\bm(\ct_\theta)$. This notation is consistent
with the definition of $\sigma_\theta$ in Section \ref{sec:pruning}. 

Using Theorem 1.1 of \cite{adv:plcrt} (see also Proposition 5.4 of
\cite{ad:ctvmp}), we get that under $\N$, the pruned tree $\ct_\theta$
is distributed as a L\'evy tree associated with the branching
mechanism $\psi_\theta$ defined by:
\[
\psi_\theta(u)=\psi(u+\theta)-\psi(\theta).
\]
Moreover, using Lemma 3.8 of \cite{ad:ctvmp}, we have the following
Girsanov formula that links the law of $\ct_\theta$ with that of
$\ct$: for every nonnegative measurable functional $F$ on the space of
trees, we have:
\begin{equation}
\label{eq:Girsanov}
\N[F(\ct_\theta)]=\N\left[F(\ct)\expp{-\alpha\theta^2\sigma}\right].
\end{equation}

Let $n$  be a positive integer.  We consider under  $\N$ (or $\N^{(r)}$)
conditionally given the tree  $\ct$, $n$ leaves $x_1,\ldots,x_n$ i.i.d.,
uniformly chosen on the set of leaves, i.e. sampled with the probability
$\bm(dx)/\sigma$.   For  $\theta>0$,  let   $(\ct^j,  j\in  J)$  be  the
connected   components   of   $\ct   \backslash   \ct_\theta$.   We write $Y_\theta^0(n)=\sum_{\ell=1}^n\ind_{\{x_\ell\in
  \ct_\theta\}}$   the  number   of   chosen  leaves   on  the   subtree
$\ct_\theta$,    and   for    $k\geq   1$,    $Y_\theta^k(n)=\Card\{   j\in
J;\sum_{\ell=1}^n\ind_{\{x_\ell\in \ct^j\}} =k\}$ the number of subtrees
with exactly $k$ chosen leaves.  In particular, we have:
\begin{equation}\label{eq:M-Y}
Y_\theta^0(n)+\sum_{k\geq     1}    k
Y_\theta^k(n)=n.
\end{equation}

We set $N_\theta(n)=\sum_{k\geq 0} Y^k_\theta(n)$ the number of chosen leaves
on $\ct_\theta$ plus the number of subtrees with chosen leaves. For
convenience, we shall consider:
\begin{equation}\label{eq:def-Ytheta}
Y_\theta(n)=\sum_{k\geq 2} Y^k_\theta(n)= N_\theta(n) -Y_\theta^0 (n)-Y_\theta^1(n). 
\end{equation}

Let  $\ct^n$ be  the reduced  tree  of the  chosen leaves:  that is  the
smallest connected  component of  $\ct$ containing the  root $\emptyset$
and $\{x_i, 1\le i\le n\}$.  Let $\ct^n_\theta$ be the  reduced tree $\ct^n$
pruned at time $\theta>0$:
\[
\ct^n_\theta=\{s\in \ct^n; \, M([0,\theta]\times \lb 0,s\rb)=0\}.
\]
Notice that $N_\theta$  is the number of leaves  of $\ct^n_\theta$ (with
the convention  that the root is  not a leaf).  Define  the last pruning
event as:
\[
L_n=\inf\{\theta>0; \, N_\theta(n)=1\}.
\]
We define:
\begin{equation}
   \label{eq:BVW}
U_n=N_{L_n-}(n), \quad V_n=Y^0_{L_n-}(n)+ Y^1_{L_n-}(n),
\quad\text{and}\quad
W_n=Y^1_{L_n-}(n).
\end{equation}
We can  interpret $U_n$  as the  number of leaves  of the  pruned reduced
tree, $V_n-W_n$  as the  number of chosen  leaves of the  pruned reduced
tree and $W_n$  as the number of subtrees with only  one chosen leaf just
before  the last  pruning  event. 

\subsubsection{Proof of Proposition \ref{prop:Ext-B}}
\label{sec:proof-VB} 

Let $B_n$ be the number of  blocks and $E_n$ be the number of singletons
involved in the last  coalescent event of $(\Pi^{[n]}(t),t\ge 0)$. Using
the link with the pruning of CRT from the previous Section, we have that
$(B_n, E_n)$ is  distributed as $(U_n, V_n)$ under   $\N^{(1)}$.

Following Remark  \ref{rem:pruning-X}, we  can see $V_n$  as the  sum of
$V_n-W_n$  (number of leaves  of $\ct^n$  with no  mark before  the last
pruning event) and the number $W_n$ of leaves of $\ct^n$ with no mark on
their  ancestral  lineage  until  the  mark corresponding  to  the  last
coalescent pruning but for the  external branch, where there is at least
one mark.

Before giving  the asymptotic  distribution of $(U_n,  V_n)$, we  need to
introduce some notations.

For $a\geq 0$, $b\geq 0$ such that $a+b>0$ we define $\Delta_0(a,b)$ as:
\begin{equation}
   \label{eq:D0}
\Delta_0(a,b)=\begin{cases}
 \displaystyle  \inv{\sqrt{\val{1+b-2a}}} \log\left(\frac
{1+\sqrt{b} + \sqrt{\val{1+b-2a}}}
{1+\sqrt{b} - \sqrt{\val{1+b-2a}}}
\right)&\quad\text{if}\quad 1+b-2a\neq 0,\\
\displaystyle \frac{2}{1+\sqrt{b}}&\quad\text{if}\quad 1+b-2a=0.
\end{cases}
\end{equation}
It is easy to check that the function $\Delta_0$ is continuous in
$(a,b)$ and that:
\begin{equation}
   \label{eq:lim-D0}
\lim_{(a,b)\rightarrow (0,0)} \Delta_0(a,b)+ \log(\sqrt{b}+a)=\log(2). 
\end{equation}
We set for $a\geq 0$ and $b\geq 0$ (with the convention $I(0,0)=1$):
\begin{equation}
   \label{eq:def-I}
I(a,b)=1+ \log(2) - \log(\sqrt{b}+a) -\Delta_0(a,b). 
\end{equation}
Notice that the function $I$ is continuous on $[0,+\infty )^2$ (and in
particular at $(0,0)$).

We shall prove the next result in   Section
\ref{sec:proof-UVW}. 
\begin{prop}
   \label{prop:UVW}
   Under  $\N^{(1)}$,  as $n$  goes  to  infinity,  $(U_n ,  V_n,  V_n)$
   converges,   in   distribution,   to   a   finite   random   variable
   $(U,V,W)$. And the distribution of $(U,V,W)$ is characterized by the
   following generating function, for $\rho,\rho_0, \rho_1\in [0,1]$: 
\[
\Psi(\rho,\rho_0,\rho_1)=\E\left[\rho^{U-V}\rho_0^{V-W}\rho_1^{W}\right]
=\rho I\left(1 -\frac{\rho+\rho_1}{2},
  1-\rho_0\right).
\]
\end{prop}
Notice that the random variable $U-V-1$ is distributed as $W$. 
Since $W\leq V$, this implies that $U-V-1$ is stochastically
smaller than $V$. This last remark and Proposition \ref{prop:UVW}
readily implies Proposition \ref{prop:Ext-B}. 

The generating function of $V-W$ is given, for
$\rho_0\in [0,1]$ by $\E\left[\rho_0^{V-W}\right]=I(0, 1- \rho_0)$, that
is: 
\[
\E\left[\rho_0^{V-W}\right]=
1+\log(2) -\log\left(\sqrt{1-\rho_0} \right)
- \inv{\sqrt{2-\rho_0}} \log\left(\frac
{1+\sqrt{1-\rho_0} + \sqrt{2-\rho_0}}
{1+\sqrt{1-\rho_0} - \sqrt{2-\rho_0}}
\right).
\]
The generating function of $W$ is given by, for
$\rho_1\in [0,1]$ by: 
\[
\E\left[\rho_1^{W}\right]=I\left(\frac{1-\rho_1}{2},0\right)
=
1+\log(4) -\log\left({1-\rho_1} \right)
- \inv{\sqrt{\rho_1}} \log\left(\frac
{1+\sqrt{\rho_1} }
{1-\sqrt{\rho_1}}
\right).
\]

\section{Proof of Proposition \ref{prop:UVW}}
\label{sec:proof-UVW}
 
In order to compute the generating function that
 appears in Proposition \ref{prop:UVW}, it is easier to work first
 under $\N$ (and then condition on $\sigma$ to have the result under $\N^{(1)}$ and to consider a Poissonian number of chosen leaves. Let
 $\lambda>0$. Under $\N$ or $\N^{(r)}$, conditionally given the tree
 $\ct$, we consider a Poisson point measure $\cn=\sum_{i\in
   I}\delta_{x_i}$ on $\ct$ with intensity $\lambda\bm$. 
   We denote by  $\tilde N=\cn(\ct)$ the
number of chosen  leaves. The law of the total mass $\sigma=\bm(\ct)$ of
$\ct$ under $\N$ is given by the following Laplace transform:
\begin{equation}
\label{eq:trans-lap-sigma}
\N[1-\expp{-\kappa\sigma}]=\psi^{-1}(\kappa)=\sqrt{\kappa/\alpha}.
\end{equation}
Conditionally  on $\sigma$,  the random variable
$\tilde N$ is  Poisson  with parameter $\lambda
\sigma$. 
Therefore, by first conditioning on $\sigma$, we get  for $k\geq 1$:
\begin{equation}
   \label{eq:moment-M}
\N[\tilde N=k]=\frac{\lambda^k}{k!} \N\left[\sigma^k \expp{-\lambda
    \sigma}\right]
=  \inv{2\sqrt{ \pi}} \sqrt{\frac{\lambda}{\alpha}} \;
\frac{\Gamma\left(k-\inv{2}\right) 
}{\Gamma(k+1)} \cdot
\end{equation}

Let $\theta>0$. 
From  the special Markov  property, Theorem  5.6 in  \cite{ad:ctvmp}
(see also  \cite{adh:crt-d}),  we  get  that  under  $\N$,  conditionally  on
$\sigma_\theta$,  the  random  variables  $(Y_\theta^k(\tilde N),  k\ge 0)$  are
independent,   $Y_\theta^0(\tilde N)$   is   Poisson   with   parameter   $\lambda
\sigma_\theta$ and for $k\geq 1$, $Y^k_\theta(\tilde N)$ is Poisson with parameter
$2\alpha \theta \sigma_\theta \N[\tilde N=k]$.

For $a\in [0,1]$, we set:
\[
f_\theta(a)=\N\left[a^{\tilde N}\rho^{Y_\theta(\tilde N)
  }\rho_0^{Y_\theta^0(\tilde N)} \rho_1^{ 
    Y_\theta^1(\tilde N)}\ind_{\{\tilde N>0\}}\right]. 
\]
\begin{lem}
   \label{lem:calcul-f}
We have:
\begin{equation}
   \label{eq:fq}
f_\theta(a)
=\theta+\sqrt{\frac{\lambda}{\alpha}} -
 \sqrt{\delta_0+2\delta_1 \sqrt{1-a} -\delta_2 a},
\end{equation}
with
\[
\delta_0
=\theta^2+\frac{\lambda}{\alpha}+ 2\theta \sqrt{\frac{\lambda}{\alpha}}
(1-\rho) , \quad
\delta_1
=\theta \rho
\sqrt{\frac{\lambda }{\alpha}}, \quad
\delta_2
=\frac{\lambda}{\alpha} \rho_0 - \theta
  \sqrt{\frac{\lambda}{\alpha}}(\rho-\rho_1).   
\]
\end{lem}
Notice that:
\begin{equation}
   \label{eq:d0-d2}
\delta_0-\delta_2= \theta^2+ \frac{\lambda}{\alpha} (1-\rho_0)+ \theta
\sqrt{\frac{\lambda}{\alpha}} (2-\rho-\rho_1)\geq \theta^2> 0.
\end{equation}
Consequently the right hand side in \reff{eq:fq} is well defined.
\begin{proof}
  We set $\mu=-\log(\rho)$  and $\mu_i=-\log(\rho_i)$ for $i\in \{0,1\}$
  and $\kappa=-\log(a)$.  We have:
\begin{multline*}
      f_\theta(a)\\
\begin{aligned}
&=\N\left[\expp{ -\mu_0 Y_\theta^0(\tilde N) -\mu_1
    Y_\theta^1(\tilde N)-\mu Y_\theta(\tilde N) - \kappa \tilde N}-
  \ind_{\{\tilde N=0\}}\right]\\
&=\N\!\!\left[\!\N\!\left[\expp{-(\mu_0+\kappa)Y_\theta^0(\tilde N)}|\sigma_\theta\right]
 \tilde \N\!\left[\expp{-(\mu_1+\kappa) 
    Y_\theta^1(\tilde N)}|\sigma_\theta\right]\prod_{k\ge 2}\N\!\left[\expp{-(\mu+k\kappa)
         Y_\theta^k(\tilde N)}|\sigma_\theta\right]\right]\!\!-\!\N[\tilde N=0],
\end{aligned}
\end{multline*}
using Equation \reff{eq:M-Y} and the independence of the variables
$Y^k_\theta(\tilde N)$ conditionally given $\sigma_\theta$.
Now, as the variables $Y^k_\theta(\tilde N)$ are conditionally given
$\sigma_\theta$ Poisson variables, and thanks to \reff{eq:moment-M},
we have:
\[
f_\theta(a)=\N\left[\expp{-\gamma\sigma_\theta}-\expp{-\lambda\sigma}\right],
\]
with
\begin{align*}
   \gamma
&=
\lambda (1-\expp{-(\kappa+\mu_0)}) 
+2\alpha \theta 
\lambda \N\left[\sigma\expp{-\lambda \sigma}\right] (1-\expp{-(\kappa+\mu_1)})
+ 2\alpha\theta  \sum_{k\geq 2} \frac{\lambda^k}{k!}
\N\left[\sigma^k\expp{-\lambda 
  \sigma}\right]  (1- \expp{-(k\kappa +\mu)}) \\
&=
\lambda (1-\expp{-(\kappa+\mu_0)}) 
+2\alpha \theta 
\lambda \N\left[\sigma\expp{-\lambda \sigma}\right] (1-\expp{-(\kappa+\mu_1)})\\
&\hspace{.5cm}+ 2\alpha\theta \left(\N\left[1-\expp{-\lambda
      \sigma}\right](1-\expp{-\mu}) - \lambda
  \N\left[\sigma \expp{-\lambda \sigma} \right](1-\expp{-(\mu+\kappa)}) 
+  \expp{-\mu} \N\left[1-
    \expp{-\lambda(1-\expp{-\kappa})\sigma} \right] 
\right).
\end{align*}
We now use Formula \reff{eq:trans-lap-sigma}
to get
\begin{align*}
\gamma &=
\lambda (1-\expp{-(\kappa+\mu_0)}) 
+\theta \sqrt{\lambda \alpha} \expp{-\kappa}
(\expp{-\mu}-\expp{-\mu_1}) + 2\theta \sqrt{\lambda\alpha}
(1-\expp{-\mu})  + 2\theta\expp{-\mu}
\sqrt{\lambda\alpha}\sqrt{1-\expp{-\kappa}}\\
&= 2\theta \sqrt{\lambda\alpha}
(1-\expp{-\mu}) +\lambda + 2\theta \expp{-\mu} \sqrt{\lambda\alpha}
\sqrt{1-a}\,  - a  \left(\lambda \expp{-\mu_0} - \theta
  \sqrt{\lambda\alpha}(\expp{-\mu}-\expp{-\mu_1})  \right).
\end{align*}

Using the  special Markov property  of \cite{ad:ctvmp}, Theorem  5.6, we
get that,  conditionally given $\sigma_\theta$,  $\sigma$ is distributed
as $\sigma_\theta+\sum\sigma_i$ where the  $\sigma_i$'s are the atoms of
a        Poisson         point        measure        of        intensity
$2\alpha\theta\sigma_\theta\N[d\sigma]$. This yields:
\[  f_\theta(a)
=\N\left[\expp{-\gamma \sigma_\theta} - \expp{-\sigma_\theta( \lambda +
  2\theta\sqrt{\lambda\alpha})}\right].
\]
To conclude, we use Girsanov Formula \reff{eq:Girsanov} to get:
\begin{align*}
f_\theta(a)&=\N\left[\expp{-(\gamma+\alpha\theta^2) \sigma} -
  \expp{-\sigma(\alpha\theta^2+ \lambda + 
  2\theta\sqrt{\lambda\alpha})}\right]\\
&=\theta+\sqrt{\frac{\lambda}{\alpha}} -
\sqrt{\theta^2+\frac{\gamma}{\alpha}}. 
\end{align*}
This gives the result. 
\end{proof}


Let us set
\[
A_n=\N^{(1)} \left[\expp{-\mu(U_n-V_n) -\mu_0(V_n-W_n)- \mu_1W_n}\right]. 
\]

To prove Proposition \ref{prop:UVW}, it is enough to prove that: 
\begin{equation}
   \label{eq:limAn}
\lim_{n\rightarrow+\infty } A_n=\rho I\left(1 -\frac{\rho+\rho_1}{2},
  1-\rho_0\right).
\end{equation}

As the CRT is coded by a Brownian excursion, it enjoys a scaling
property, namely the law of $r\ct$ under $\N^{(r^2)}$ is those of
$\ct$ under $\N^{(1)}$ (where $r\ct$ means that we multiply the
distance on the tree by a factor $r$). Consequently, the mark process
(defined as a Poisson point measure with intensity proportional to the
length measure) also satisfies a scaling property. It is then easy to
deduce that the law of $(U_n, V_n, W_n)$ doesn't depend on
$\sigma$. So, we have:
\[
A_n=\N^{(r)}\left[\expp{-\mu(U_n-V_n) -\mu_0(V_n-W_n)- \mu_1W_n}\right]
\]
for every positive $r$.
We set:
\begin{equation}
   \label{eq:derf-An}
\ca_n=A_n\N\left[\sigma^n \expp{-\lambda \sigma}\right].
\end{equation}
By conditioning on $\sigma$, we get:
\[
\ca_n=
 \N\left[\sigma^n \expp{-\lambda
    \sigma} \expp{-\mu (U_n-V_n) -\mu_0(V_n-W_n) -\mu_1 W_n} \right].
\]
For $n\geq 1$, we set:
\begin{equation}
   \label{eq:F-Hn}
F_\theta(n,r)=\N^{(r)}\left[\rho^{Y_\theta(n)}\rho_0^{Y_\theta^0(n)}
  \rho_1^{Y_\theta^1(n)}\right]\quad \text{and}\quad
H_n(\theta)=\N\left[F_\theta(n,\sigma) \sigma^n \expp{-\lambda \sigma
  }\right]. 
\end{equation}
Recall:
\begin{equation}
   \label{eq:Mn}
\N\left[\sigma^n \expp{-\lambda \sigma}\right]= \lambda^{-n} \, n!\,
\N\left[\tilde N=n\right] = \inv{2\sqrt{\alpha \pi} \lambda^{n-\inv{2 }}}
\Gamma\left(n-\inv{2}\right) .
\end{equation}
Therefore, we have:
\[
f_\theta(a)=\N\left[a^{\tilde N} F_\theta(\tilde N,\sigma)\ind_{\{
    \tilde N>0\}}\right],
\]
and, thanks to \reff{eq:Mn},  for $n\geq 1$:
\[
f^{(n)}_\theta(0)=n! \N\left[F_\theta(n,\sigma) \ind_{\{\tilde N=n\}}\right]
= \lambda^n H_n(\theta). 
\]

We use the description of \cite{dl:rtlpsbp} of the reduced tree
spanned by $n$ leaves under the $\sigma$-finite measure
$\N$: it is a uniform binary tree with $n$ leaves and with edge
lengths i.i.d. and ``distributed'' as $\alpha dh$. We denote by $\lb
\emptyset, x_1\rb$ the edge of the reduced tree attached to the root,
and by $H$ its length. The time $L_n$ at which the last coalescent event
occurs is just the first time $\theta$ at which a mark appears on $\lb
\emptyset, x_1\rb$ and therefore it is, conditionally given $H=h$,
exponentially distributed with parameter $2\alpha h$. Moreover, if we denote by
$\ct^{(1)}$ and $\ct^{(2)}$ the two sub-trees attached to $x_1$, and
$\sigma^{(1)}$, $\sigma^{(2)}$ respectively their total mass, we
have: 
\[
\sigma=\sigma^{(1)}+\sigma^{(2)}+\sum_{i\in I'}\sigma_i,
\]
where the $\sigma_i$'s  are the total mass of  the sub-trees attached on
the   edge   $\lb   \emptyset,   x_1\rb$.  The random measure
$\sum_{i\in I'}\sigma_i$ is    independent   of
$\sigma^{(1)}$,  $\sigma^{(2)}$  and  is,  conditionally  on   $\{H=h\}$,
 distributed  as a  Poisson point  measure with  intensity
$2\alpha  h\N[d\sigma]$. Eventually, the  two reduced sub-trees  attached to
$x_1$  are independent  and  distributed as  uniform  binary trees  with
respectively $k$ and $n-k$ leaves ($0<k<n$). Recall that $2\alpha h \N\left[1-\expp{-\lambda \sigma}\right]= 2
\sqrt{\alpha \lambda} h$. We deduce from  this description that: 
\begin{align}
\nonumber
   \ca_n
&= \int_0^\infty  \alpha \expp{-2\sqrt{\alpha \lambda} h}\; dh \;  
\int_0^\infty  2\alpha h \expp{-2 \alpha h \theta}\; d\theta \; 
\sum_{k=1}^{n-1} \binom{n}{k}  H_k(\theta) H_{n-k}(\theta) \\
\label{eq:ca}
&= \lambda^{-n} \int_0^\infty 
\frac{ d\theta}{\left(\theta + 
    \sqrt{ \lambda/\alpha} \right)^2}  \; 
G_n(\theta),
\end{align}
with 
\begin{equation}
   \label{eq:G}
G_n(\theta)=\frac{\lambda^n}{2}\sum_{k=1}^{n-1} \binom{n}{k}
H_k(\theta) H_{n-k}(\theta)  
=\inv{2} \sum_{k=1}^{n-1} \binom{n}{k}  f^{(k)}_\theta(0)
f^{(n-k)}_\theta(0).
\end{equation}
Let $c_0=\theta+\sqrt{\lambda/\alpha}$  and $g=c_0-f_\theta$, so that:
\[
g(a)= \sqrt{\delta_0+2\delta_1 \sqrt{1-a} -\delta_2 a}.
\]
If $h$ is a function, we denote $\partial^n_{a=0}h(a)$ for
$h^{(n)}(0)$. Then, using the formula for the $n$-th derivative of a
product of functions,
we have for $n\geq 2$:
\[\
G_n(\theta)
=\inv{2} \sum_{k=1}^{n-1} \binom{n}{k}  g^{(k)}(0)
g^{(n-k)}(0)
= \inv{2} \partial^n_{a=0}\; 
  g^2(a) -  g(0) g^{(n)}(0) .
\]
That is, since $\partial^n_{a=0}
\sqrt{1-a}=-{\Gamma\left(n-\inv{2}\right)}/{2\sqrt{\pi}}$:
\begin{equation}
   \label{eq:exp-G}
G_n(\theta)
= \delta_1 \partial^n_{a=0}
\sqrt{1-a}- \sqrt{\delta_0+2\delta_1}\,
g^{(n)}(0)
=- \delta_1 \frac{\Gamma\left(n-\inv{2}\right)}{2\sqrt{\pi}} - \sqrt{\delta_0+2\delta_1}\,
g^{(n)}(0). 
\end{equation}

The next Lemma gives an equivalent expression for $g^{(n)}(0)$. 
\begin{lem}
   \label{lem:gn(0)}
We have:
\[
\lim_{n\rightarrow+\infty } -\frac{g^{(n)}(0)}
{\Gamma\left(n-\inv{2}\right)} = \frac{\delta_1}{2\sqrt{\pi}}
\inv{\sqrt{\delta_0 -\delta_2}}.
\]
\end{lem}
\begin{proof}
Using that the density, which corresponds to the density of the $1/2$-stable
subordinator with no drift:
\[
h(x)= \frac{\delta_1 r}{\sqrt{\pi}} \frac{1}{x^{3/2}} \expp{-
  \delta_1^2 r^2/x}\ind_{\{x>0\}}
\]
has Laplace transform:
\[
\int_0^{+\infty } \expp{-\lambda x} h(x)\,dx=\expp{-2\delta_1 r
  \sqrt{\lambda}},
\]
we can write:
\begin{align*}
   g(a)
&=\inv{2\sqrt{\pi}}\int_0^\infty  \frac{dr}{r^{3/2}} 
\left(1- \expp{-(\delta_0 +2 \delta_1 \sqrt{1-a} -\delta_2 a
    )r}\right)\\
&=\inv{2\pi}\int_0^\infty  \frac{dr}{r^{3/2}} \int_0^\infty 
\delta_1 r \frac{dx}{x^{3/2}} \expp{- \delta_1^2 r^2/x}
\left(1- \expp{-\delta_0r -x +a (\delta_2 r+x)}\right). 
\end{align*}
We deduce, with $y=r/x$ that for $n\geq 1$:
\begin{align*}
  - g^{(n)}(0)
&=\inv{2{\pi}}\int_0^\infty  \frac{dr}{r^{3/2}} 
 \int_0^\infty 
\delta_1 r \frac{dx}{x^{3/2}}  \; (\delta_2 r + x )^n 
\expp{-\delta_0r -x - \delta_1^2 r^2/x}\\
&=\frac{\delta_1}{2\pi} \int_0^\infty  r^{n-1} dr\int_0^\infty
\frac{dy}{\sqrt{y}} \left(\frac{1+\delta_2 y}{y}\right)^n  
\expp{-r(\delta_0+\inv{y} + \delta_1^2 y)} \\
&=\frac{\delta_1\Gamma(n)}{2\pi} \int_0^\infty
\frac{dy}{\sqrt{y}} \left(\frac{1+\delta_2 y}{y}\right)^n  
\inv{(\delta_0+\inv{y} + \delta_1^2 y)^n} \\
&=\frac{\delta_1\Gamma(n)}{2\pi} \int_0^\infty
\frac{dy}{\sqrt{y}} \varphi(y)^n,
\end{align*}
with
\[
\varphi(y)= \frac{1+\delta_2 y}{1+ \delta_0 y +\delta_1^2 y^2}\cdot
\]
To get an equivalent expression for $g^{(n)}(0)$, we use Laplace's method.
Notice that $\varphi(0)=1$
and $\varphi'(0)=\delta_2 -\delta_0<0$ thanks to \reff{eq:d0-d2}. 
Notice that $\delta_0>0$ and $\delta_1>0$. 

The only root to $\varphi(y)=1$ is $0$. Let us assume there is a root to
$\varphi(y)=-1$. This implies that $
1+\delta_0 y +\delta_1^2 y^2+1+\delta_2y=0$
that is:
\[
(\delta_0+\delta_2)y + \delta_1^2 y^2=-2.
\]
But we have:
\[
\delta_0+\delta_2\geq \theta^2+\frac{\lambda}{\alpha} - \theta
\sqrt{\lambda}{\alpha} >0.
\]
Therefore, $|\varphi(y)|\in (-1,1)$ for $y>0$. Since,
$\lim_{y\rightarrow+\infty } \varphi(y)=0$.

The discussion above proves that: 
$$\exists \delta_0>0,\ \forall \delta\in(0,\delta_0),\ \forall x\ge
\delta, \bigl|\varphi(x)\bigr|\le \varphi(\delta)<1.$$

Moreover, we have:
\[
\lim_{x\downarrow 0} \frac{\log(\varphi(x))}{(\delta_2-\delta_0)x}=1. 
\]
Therefore we can apply the standard Laplace's method and get
 the following asymptotic in $n$:
\[
   - g^{(n)}(0)
\sim \frac{\delta_1\Gamma(n)}{2\pi}  \int_0^\infty
\frac{dy}{\sqrt{y}}  \; \expp{ - n (\delta_0-\delta_2) y}\\
= \frac{\delta_1 \Gamma(n)}{2\pi}
\frac{\Gamma(1/2)}{\sqrt{ n(\delta_0-\delta_2) }}  \\
\sim \frac{\delta_1\Gamma\left(n-\inv{2}\right)}{2\sqrt{\pi}}
\inv{\sqrt{\delta_0 -\delta_2}},
\]
where $a_n \sim b_n$ means $\lim_{n\rightarrow+\infty } a_n/b_n=1$.    
\end{proof}

Recall the notations of Proposition \ref{prop:UVW}:
\[
\Psi(\rho,\rho_0,\rho_1)=\rho I\left(1 -\frac{\rho+\rho_1}{2},
  1-\rho_0\right).
\]

\begin{lem}
   \label{lem:cvPsi}
We have, for $\rho,\rho_0, \rho_1\in [0,1]$:
\[
\lim_{n\rightarrow +\infty } 
\N^{(1)}\left[\rho^{U_n-V_n}\rho_0^{V_n-W_n}\rho_1^{W_n}\right]
=\Psi(\rho,\rho_0,\rho_1).
\]
\end{lem}
Notice that $ \Psi(1,1,1)
= I(0,0)=1$. This implies that $(U_n, V_n, W_n)$ converge in
distribution, as $n$ goes to infinity, to an a.s. finite random variable
$(U,V,W)$ and that the generating function of $(U-V, V-W,W)$ is given by
$\Psi$. 

\begin{proof}
On the one hand, we deduce from \reff{eq:exp-G} and Lemma
\ref{lem:gn(0)} that:
\[
\lim_{n\rightarrow+\infty } \frac{G_n(\theta)}{\Gamma\left(n-\inv{2}\right)}
=-\frac{\delta_1 }{2\sqrt{\pi}} +
\sqrt{\delta_0+2\delta_1}
\frac{\delta_1}{2\sqrt{\pi}} 
\inv{\sqrt{\delta_0 -\delta_2}}= \frac{\delta_1 }{2\sqrt{\pi}} 
\left(  \frac{\sqrt{\delta_0+2\delta_1}}{\sqrt{\delta_0 -\delta_2}}
-1\right).
\]
On the other hand, we deduce from \reff{eq:F-Hn} (by considering
$\rho=\rho_0=\rho_1=1$) and \reff{eq:Mn} that:
\begin{equation}
   \label{eq:majo-Hn}
H_n(\theta)\leq   \N\left[\sigma^n \expp{-\lambda \sigma} \right]
= \inv{2\sqrt{\alpha \pi} \lambda^{n-\inv{2 }}}
\Gamma\left(n-\inv{2}\right). 
\end{equation}
Recall \reff{eq:def-C}. 
By decomposing an ordered binary tree with $n$
labelled leaves in   two ordered binary sub-trees attached
to  closest node of the
root, we get:
\[
\sum_{k=1}^{n-1} \binom{n}{k} C_k C_{n-k} =C_n. 
\]
This readily implies:
\[
\sum_{k=1}^{n-1} \binom{n}{k}
\Gamma\left(k-\inv{2}\right)\Gamma\left(n-k-\inv{2} \right)
= 4\sqrt{\pi} \Gamma\left(n-\inv{2}\right).
\]
We deduce from the first equality of \reff{eq:G} and \reff{eq:majo-Hn} that:
\[
G_n(\theta)=\frac{\lambda^n}{2}\sum_{k=1}^{n-1} \binom{n}{k}
H_k(\theta) H_{n-k}(\theta) 
\leq  \frac{\lambda}{2\alpha \sqrt{\pi}}
\Gamma\left(n-\inv{2}\right). 
\]
This implies that:
\[
\frac{G_n(\theta)}{\Gamma\left(n-\inv{2}\right)}\leq
\frac{\lambda}{2\alpha \sqrt{\pi}}\cdot
\]
By dominated convergence, we deduce  from \reff{eq:ca} that: 
\begin{align*}
\lim_{n\rightarrow+\infty } \frac{\lambda^n \ca_n}
 {\Gamma\left(n-\inv{2}\right)}
&= \int_0^\infty 
\frac{ d\theta}{\left(\theta + 
    \sqrt{ \lambda/\alpha} \right)^2}  \; 
\lim_{n\rightarrow+\infty }  \frac{G_n(\theta)} 
 {\Gamma\left(n-\inv{2}\right)}\\
&=
\frac{\sqrt{\lambda}}{2\sqrt{\alpha\pi}}   \rho
\int_0^\infty \!\!\!
\frac{ \theta d\theta}{\left(\theta + 
    \sqrt{ \lambda/\alpha} \right)^2}  \; 
\left(
 \frac{\theta+\sqrt{\frac{\lambda}{\alpha}}}
 {\sqrt{ \theta^2 
 + \theta \sqrt{\frac{\lambda}{\alpha}}(2 -\rho - \rho_1)
 + \frac{\lambda}{\alpha}(1- \rho_0)}}
-1 \right)   \\
&=
\frac{\sqrt{\lambda}}
{2\sqrt{\alpha \pi}} \;  \rho
\int_0^\infty 
\frac{ \theta d\theta}{\left(\theta + 1 \right)^2}  \; 
\left(
 \frac{\theta+1}
 {\sqrt{ \theta^2 
 + \theta (2 -\rho - \rho_1)
 + (1-\rho_0)}}
-1 \right)  . 
\end{align*}
We deduce from \reff{eq:derf-An} that:
\begin{align*}
\lim_{n\rightarrow +\infty } A_n
&=
\frac{2\sqrt{\alpha \pi}}{\sqrt{\lambda}}  \lim_{n\rightarrow+\infty }
\frac{\lambda^n \ca_n} 
 {\Gamma\left(n-\inv{2}\right)}\\
&=
\rho \int_0^\infty 
\frac{ \theta d\theta}{\left(\theta + 1 \right)^2}  \; 
\left(
 \frac{\theta+1}
 {\sqrt{ \theta^2 
 + \theta (2 -\rho - \rho_1)
 + (1-\rho_0)}}
-1 \right)  .
\end{align*}
Then, we use Lemma \ref{lem:I} below to conclude. 
\end{proof}
Before stating and proving Lemma \ref{lem:I}, we first give a
preliminary result. 
\begin{lem}
   \label{lem:Delta}
For $a\geq 0$, $b\geq 0$ and $a+b>0$, we have
$\Delta(a,b)=\Delta_0(a,b)$. 
\end{lem}
\begin{proof}
We first assume $1+b-2a>0$. From \reff{eq:Dab}, we have:
\begin{align*}
   -\partial_a \Delta(a,b)
&= \int_0^\infty
\frac{\theta\;d\theta}{(\theta+1)( \theta^2+2a \theta+b)^{3/2}}\\
&= -2 \partial_bI(a,b)\\
&= -\frac{\Delta(a,b)}{1+b-2a} +
\inv{1+b-2a} \frac{\sqrt{b}+1}{\sqrt{b}+a}\cdot
\end{align*}
Then, by computing the derivative $\partial_a \Delta_0$, we
deduce that $\Delta(a,b)=\Delta_0(a,b) + h_b(a)$, for a
 function $h_b$ solving:
\[
-h'_b(a)= -\frac{h_b(a)}{1+b-2a},
\]
that is, for some function $c_+$:
\[
h_b(a)=\frac{c_+(b)}{\sqrt{1+b-2a}}\cdot
\]
Similarly, we have that for $1+b-2a<0$,
\[
\Delta(a,b)=\Delta_0(a,b) + \frac{c_-(b)}{\sqrt{\val{1+b-2a}}},
\]
for some function $c_-$. Notice that $\Delta$ and $\Delta_0$ are by
definition \reff{eq:Dab} and \reff{eq:D0} continuous on $(0,+\infty
)^2$. By letting $a$ goes to $(1+b)/2$, we deduce that $c_+=c_-=0$. 
This proves the result. 
\end{proof}

Recall the definition of $I$ given in \reff{eq:def-I}. We set for $a\geq
0$, $b\geq 0$:
\[
J(a,b)
=\int_0^\infty \frac{  d\theta}{ \theta + 1 }  \; 
\left( \frac{\theta} {\sqrt{ \theta^2   + 2a\theta + b}}
-\frac{\theta}{\theta+1} \right)   .
\]
\begin{lem}
   \label{lem:I}
For $a\geq 0$, $b\geq 0$, we have $J(a,b)=I(a,b)$. 
\end{lem}

\begin{proof}
We first notice that:
\begin{equation}
   \label{eq:3/2}
\int_0^\infty \frac{\theta\;
  d\theta}{(\theta^2+2a\theta+b)^{3/2}}=\inv{\sqrt{b}+a}
\quad\text{and}\quad
\int_0^\infty \frac{
  d\theta}{(\theta^2+2a\theta+b)^{3/2}}=\inv{\sqrt{b}(\sqrt{b}+a)}\cdot
\end{equation}

For $a\geq 0$, $b\geq 0$ and $a+b>0$, we set:
\begin{equation}
   \label{eq:Dab}
\Delta(a,b)=\int_0^\infty
\frac{d\theta}{(\theta+1)\sqrt{\theta^2+2a\theta+b}}. 
\end{equation}
We have:
\begin{align}
\nonumber
-   \partial_bJ(a,b)
&=\inv{2} \int_0^\infty
\frac{\theta\;d\theta}{(\theta+1)( \theta^2+2a \theta+b)^{3/2}}\\
\nonumber
&=\inv{2(1+b-2a)} \int_0^\infty
\frac{d\theta}{\sqrt{ \theta^2+2a \theta+b}}
\left[\frac{\theta+b}{\theta^2+2a \theta+b}-\inv{\theta+1} \right]\\
\nonumber
&=-\frac{\Delta(a,b)}{2(1+b-2a)} +
\inv{2(1+b-2a)} \left[\inv{\sqrt{b}+a} 
+ \frac{b}{\sqrt{b}(\sqrt{b}+a)} \right]\\
\label{eq:dIa}
&= -\frac{\Delta(a,b)}{2(1+b-2a)} +
\inv{2(1+b-2a)} \frac{\sqrt{b}+1}{\sqrt{b}+a}\cdot
\end{align}
We also have:
\begin{align}
\nonumber
-   \partial_aJ(a,b)
&= \int_0^\infty
\frac{\theta^2\;d\theta}{(\theta+1)( \theta^2+2a \theta+b)^{3/2}}\\
\nonumber
&=\inv{1+b-2a} \int_0^\infty
\frac{d\theta}{\sqrt{ \theta^2+2a \theta+b}}
\left[\frac{(b-2a)\theta-b}{\theta^2+2a \theta+b}+\inv{\theta+1} \right]\\
\nonumber
&=\frac{\Delta(a,b)}{1+b-2a} +
\inv{1+b-2a} \left[\frac{b-2a}{\sqrt{b}+a} 
- \frac{b}{\sqrt{b}(\sqrt{b}+a)} \right]\\
\label{eq:PaI}
&= \frac{\Delta(a,b)}{1+b-2a} +
\inv{1+b-2a} \frac{b-\sqrt{b}-2a}{\sqrt{b}+a}\cdot
\end{align}

After computing $\partial_a\Delta(a,b)$, we deduce from \reff{eq:PaI}
that, for $a+b>0$:
\[
J(a,b)
= -\Delta(a,b) - \log\left(\sqrt{b}+a\right) + g(b),
\]
for some function $g$. Then  computing $\partial_b J(a,b)$, we get using
\reff{eq:dIa} that  $g(b)$ is  a constant $c$.  Eventually, on  the one
hand taking $a=\sqrt{b}=1$, we get:
\[
   J(1,1)
=\int_0^\infty
\frac{d\theta}{\theta+1}\left[\frac{\theta}{\theta+1} -
  \frac{\theta}{\theta+1} \right]=0.
\]
On the other hand, we have:
\[
J(1,1)
= -\Delta(1,1) - \log\left(2\right) + c 
= -1 - \log(2) +c. 
\]
This gives $c=1+\log(2)$. We get that:
\[
J(a,b)= -\Delta(a,b) - \log\left(\sqrt{b}+a\right) + 1+\log(2).
\]
That is $I=J$ for $a+b>0$. Then, we use the continuity of $I$ and $J$ to get
$I=J$.
\end{proof}

\newcommand{\sortnoop}[1]{}

\end{document}